\documentclass[a4paper,11pt]{amsart}
\usepackage{tikz,tikz-cd,amscd,amsmath, amsthm,amssymb,braket,mathrsfs,graphicx,verbatim,mathtools}
\usepackage{todonotes}
\usepackage{hyperref}
\hypersetup{colorlinks=true, linkcolor=black,citecolor=black}
\usepackage{subcaption}
\usepackage[margin=2.5cm]{geometry}

\usetikzlibrary{matrix,arrows,decorations.pathmorphing}
\usepackage[all]{xy}
\newtheorem{theorem}{Theorem}[section]
\theoremstyle{definition}\newtheorem*{definition}{Definition}
\theoremstyle{definition} \newtheorem{proposition}[theorem]{Proposition}
\theoremstyle{definition}\newtheorem{lemma}[theorem]{Lemma}
\theoremstyle{definition}\newtheorem{corollary}[theorem]{Corollary}
\theoremstyle{definition}
\theoremstyle{remark}\newtheorem*{remark}{Remark}
\usepackage{xcolor}
\definecolor{steel}{HTML}{396a93}

\newcommand{\IM}{\imm^{1}(\S,\R^2)}

\def\phat{\hat{\partial} }

\def\R{\mathbb{R}}

\def\S{\mathbb{S}}

\def\x{\mathsf{x}}

\def\h1{{H^1}}
\def\l2{{L^2}}
\def\G{\mathcal G^\lambda}

\newcommand{\norm}[1]{{\left \Vert {#1}\right \Vert}}
\newcommand{\enorm}[1]{{\left | {#1}\right |}}
\newcommand{\ip}[1]{{\left \langle {#1} \right \rangle}}

\newcommand{\abs}[1]{\left | {#1}\right |}

\DeclareMathOperator{\imm}{Imm}

\DeclareMathOperator{\sgn}{sgn}

\def\length{\mathcal L}
\renewcommand{\L}{\length}

\def\hpartial{\hat\partial}
\def\wimm{W^{1,1}_{\imm}(\S,\R^2)}
\def\wsr{W^{1,1}(\S,\R^2)}

\begin{document}
\title{Homogeneous Sobolev gradient flow of the length functional}

\author[P. Schrader]{Philip Schrader}
\address{School of Mathematics and Statistics, Chemistry and Physics\\ Murdoch University\\
South Street, Murdoch, WA $6150$\\ Australia}
\email{phil.schrader@murdoch.edu.au}

\author[G. Wheeler]{Glen Wheeler}
\address{Institute for Mathematics and its Applications, School of Mathematics and Applied Statistics\\
University of Wollongong\\
Northfields Ave, Wollongong, NSW $2500$\\
Australia}
\email{glenw@uow.edu.au}

\author[V.-M. Wheeler]{Valentina-Mira Wheeler}
\address{Institute for Mathematics and its Applications, School of Mathematics and Applied Statistics\\
University of Wollongong\\
Northfields Ave, Wollongong, NSW $2500$\\
Australia}
\email{vwheeler@uow.edu.au}

\thanks{
This research was supported in part by
projects FT250100880, DP250101080 and DE190100379 of the Australian Research
Council.}

\subjclass{53C44}

\begin{abstract}
The well-known curve shortening flow can be formulated as the gradient flow of the length functional on the space of immersed closed planar curves, where the gradient is taken with respect to a reparametrisation-invariant $L^2$ Riemannian metric. This metric is degenerate, giving a geodesic distance of zero between any two curves. We instead consider a family of Sobolev $H^1$ metrics depending on two parameters $\lambda>0$ and $a\in \R$, where $\lambda$ sets the weight of the first-derivative term, and $a$ indexes a length normalisation which ensures that the metric is scale-homogeneous. For each such metric, the gradient of length can be written explicitly in terms of a convolution with respect to normalised arc length against the periodic Green's function of $(\lambda^2 \partial_x^2-1)$. The associated evolution is a reparametrisation invariant nonlocal ODE whose right-hand side is well-defined even on curves that are not immersed. Working in the optimal low-regularity setting $W^{1,1}(\S,\R^2)$, we prove local well-posedness using the Picard--Lindel\"of theorem and convergence to constant maps in finite time when $a<2$, and as $t\to\infty$ when $a\geq 2$. This behaviour is exhibited by round circles, which evolve self-similarly and collapse at an explicit time. We further prove that if the initial curve is an immersion, $C^1$, $C^2$, or bounds a strictly convex set, then each of these properties is preserved along the flow. 

\end{abstract}

\maketitle

\section{Introduction}
Let $\imm^{1}(\S,\R^2)$ denote the space of $C^1$ immersions of the circle into the Euclidean plane, i.e. continuously differentiable maps $\gamma:\S\to \R^2$ with $|\gamma'|\ne0$, where $\S=[0,1]/{\sim}$ . The length functional:
\begin{equation*}
\length(\gamma):=\int_0^1 \enorm{\gamma'(u)}du \end{equation*}
is differentiable on $\imm^{1}(\S,\R^2)$ with the derivative of $\length$ at $\gamma\in \IM$ in the direction of a variation $v: \S^1 \to \R^2$ given by
\begin{equation}
\begin{split}
D\length(\gamma) v = {\left .  \partial_\varepsilon \length(\gamma+\varepsilon v) \right |}_{\varepsilon=0}
 \label{dlength}
=\int_0^1 \frac{\langle\gamma'(u),v'(u)\rangle}{\abs{\gamma'(u)}} du, 
\end{split}
\end{equation}
where the inner product on $\R^2$ is the Euclidean one.

A gradient direction $\nabla \length(\gamma)$ for the length functional at $\gamma$ is determined by a choice of inner product or Riemannian metric $\ip{\cdot,\cdot}_\gamma$ on $\IM $ according to $ D\L_\gamma=\ip{\nabla\length(\gamma), \cdot}_\gamma$. For example, in the parametrisation invariant $L^2$ metric defined by
\[
\ip{v,w}_{L^2(\gamma)}:=\int_{\S} \ip{v(u),w(u)}\, |\gamma'(u)| du,
\]
after reparametrisation and integration by parts in \eqref{dlength} the gradient is $-kN$, where $k$ is the curvature scalar and $N$ the normal vector along $\gamma$. The corresponding gradient flow is the well-known curve shortening flow (see the classical papers \cite{GH},\cite{Gage:1983aa},\cite{Gray}, and the more recent book \cite{Andrews:2020aa} Chapters 2-3).

Motivated by results of Michor and Mumford showing that the Riemannian distance in the $L^2(\gamma)$ metric is always zero \cite{MR2201275}, and that Sobolev $H^1$ type metrics do not have this pathology, the authors previously studied \cite{schrader2023h} the gradient flow of $\length$ with respect to the metric 
\[{\ip{v,w}_{L^2(\gamma)}+\ip{\tfrac{1}{\abs{\gamma'}} v',\tfrac{1}{\abs{\gamma'}} w'} _{L^2(\gamma)}}.\] 
This flow exhibits scale dependent behaviour because the two summands in the metric have different scaling under dilations of the plane. 

In this paper we consider instead the gradients of $\length$ corresponding to the Riemannian metrics:
\begin{align}\label{metrics}
\ip{v,w}_{H^1_{\lambda,a}(\gamma)}&:=\frac{1}{\length(\gamma)^{a}}\ip{v,w}_{L^2(\gamma)}+\frac{\lambda^2}{\length(\gamma)^{a-2}} \ip{\tfrac{1}{\abs{\gamma'}} v',\tfrac{1}{\abs{\gamma'}} w'} _{L^2(\gamma)}, 
\end{align}
where $\lambda > 0, a\in\R$ are parameters, and the powers of $\length$ ensure that both summands have the same scaling. Indeed if we dilate by a factor of $\alpha$  we find that 
\[
\ip{\alpha v, \alpha w}_{H^1_{\lambda,a}(\alpha \gamma)}=\alpha^{3-a}\ip{v,w }_{H^1_{\lambda,a}(\gamma)}
\]
and so these metrics are $(3-a)$-homogeneous. Two cases are particularly natural: if $a=3$ then the induced map $\alpha:\IM \to \IM$ is an isometry, and if $a=1$ the induced Riemannian distance scales by the same factor $\alpha$. The former is most interesting from the point of view of shape analysis
(see e.g.~\cite{bauer2024elastic},\cite{bauer2011sobolev},\cite{MR2201275}) , while the latter seems very natural from a geometric point of view.

In Lemma \ref{gradientlemma} we derive an explicit formula \eqref{thegradient} for the gradients of $\L$ with respect to the metrics \eqref{metrics}. The kernel $\mathcal G^\lambda(x)$ in \eqref{curlyg} is the periodic Green's function solving ${(\lambda^2 \partial_x^2-1)\mathcal G^\lambda(x)=\delta(x)}$, and the gradient reads
\[
\nabla\L_{\lambda,a}(\gamma)=\frac{\L(\gamma)^{a-2}}{\lambda^2}\Big(\gamma+\gamma\ast_\gamma \mathcal G^\lambda\Big),
\]
where the $\ast_\gamma$ convolution is defined by \eqref{convolution}. Although it is derived under the assumption that $\gamma$ is an immersion, $\nabla\L_{\lambda,a}(\gamma)$ is well-defined for any $\gamma\in W^{1,1}(\S,\R^2)$ with $\length(\gamma)\neq 0$. We therefore define:

\begin{definition}
Let $I$ be an interval containing $0$.
We say $X:\S\times I\to \R^2$ is an $H^1_{\lambda,a}(\gamma)$-gradient flow of the length functional if it satisfies
\begin{equation}\label{GF}\tag{GFa}
\partial_t X(u,t)=F_{\lambda,a}(X):= 
\begin{cases}
- \frac{\L(X)^{a-2}}{\lambda^2}\Big(X(u,t)+(X\ast_X \mathcal G^\lambda) (u,t)\Big), & \length(X)\neq 0 \\
0 & \length(X)=0
\end{cases}
\end{equation}
\end{definition}

An illuminating exact solution is given by round circles: under the flow, a circle of radius $r(t)$ remains round and $r$ solves an autonomous ODE whose decay rate depends on $a$ and $\lambda$ (Section~\ref{sec:circle}). In particular, circles shrink self-similarly to a point, with extinction in finite time for $a<2$ and in infinite time for $a\ge 2$.

From \eqref{curlyg} and \eqref{convolution} we observe that under dilation by a factor of $\alpha$ the gradient satisfies $\nabla\L_{\lambda,a}(\alpha \gamma)=\alpha^{a-1} \nabla\L_{\lambda,a}(\gamma)$. That is, the gradient and therefore also the gradient flow are $(a-1)$-homogeneous. It follows that when $a=2$ the flow is \emph{equivariant} under Euclidean scaling: if $X$ solves \eqref{GF2} then so does $\alpha X$ for any $\alpha>0$.

We treat \eqref{GF} as an ODE on the optimal low-regularity space $W^{1,1}(\S,\R^2)$: this space is
the largest in which $\L$ is well-defined and $F_{\lambda,a}$ satisfies the continuity properties needed for the Picard-Lindel\"of theorem. It turns out that the case $a=2$ is easiest to analyse, and, as we show in Proposition~\ref{timescaling}, solutions of \eqref{GF} differ from solutions to \eqref{GF2} only by a time reparametrisation. This observation allows us to focus for the most part on \eqref{GF2}, and then extend results to \eqref{GF}.

\medskip

Our main results are as follows.

\begin{theorem}\label{thm:main1}
Let $X_0\in W^{1,1}(\S,\R^2)$ with $\length(X_0)>0$, and $\lambda>0$, $a\in\R$. Then there is a unique positive-length solution $X\in C^1([0,T_a),W^{1,1}(\S,\R^2))$ to \eqref{GF} with $X(\cdot,0)=X_0$, where 
$T_a=\infty$ if $a\geq 2$, while for $a<2$, 
\[ 0<T_a \leq \frac{1+8\lambda^2}{4(2-a)}\length(X_0)^{2-a}<\infty.
\]
As $t\uparrow T_a$, $X(\cdot,t)$ converges in $W^{1,1}(\S,\R^2)$ to a constant map $u\mapsto \x_\infty \in \R^2$ with ${\abs{\x_\infty}\leq \norm{X_0}_{L^\infty}}$. 
\end{theorem}

We remark that for $1<a<2$ the positive-length solution has a stationary $C^1$ continuation after shrinking to a point, since $\norm{F_{\lambda,a}(X(t))}_{W^{1,1}}\to 0$ by Corollary \ref{Fa-bounds}. For $a=1$ this does not hold; the explicit solution for evolving circles in Section \ref{sec:circle} provides a counterexample, and shows that for $a<1$ the velocity may blow up as extinction is approached.  

As with the inhomogeneous Sobolev flow studied in \cite{schrader2023h}, if $X_0$ is $C^1$ and immersed then so is $X(\cdot, t)$ for as long as the flow exists. Our second main result is that $C^2$ regularity and convexity are also preserved.

\begin{theorem}\label{thm:main2}
Let $X_0\in W^{1,1}(\S,\R^2)$ with $\length(X_0)>0$, and $\lambda>0$, $a\in\R$, and let $X\in C^1([0,T_a),W^{1,1}(\S,\R^2))$ be the solution to \eqref{GF} with $X(\cdot,0)=X_0$. If $X_0\in C^2(\S,\R^2)$ then $X\in C^1([0,T_a),C^2(\S,\R^2))$,
and if $X_0$ is also an immersion such that $X_0(\S)$ is the boundary of a strictly convex set in $\R^2$, then $X(\S,t)$ is the boundary of a strictly convex set for all $t$. 
\end{theorem}

Theorem~\ref{thm:main2} follows by combining Proposition~\ref{timescaling}, 
Lemma~\ref{PRimmersed}, the $C^1$ and $C^2$ existence results and the convexity argument in Section~\ref{sec:convex}.

\emph{Numerics.} The explicit periodic Green’s function makes the right-hand side a convolution-type nonlocal operator with a rapidly decaying kernel for small $\lambda$. A straightforward discretisation is: sample $u$ uniformly, evaluate arc length by discrete cumulative sums of $|X'|$, and approximate the convolution by a circulant matrix multiplication. The formulation \eqref{eq:Fmod} with $\gamma(v)-\gamma(u)$ in the integrand suppresses round-off in the small-$\lambda$ regime. We have used such a numerical scheme to produce the illustrations in Figures \ref{fig1} and \ref{fig2} which appear to be sensible, but we do not make any claims of accuracy or convergence of the scheme in this paper.

In Figure \ref{fig1} we observe exponential decay of length for solutions to \eqref{GF2} with $\lambda=1$ and $\lambda=0.1$. It is clear in the case of $\lambda=0.1$ that the flow has a rounding effect, but this is not apparent for $\lambda=1$. However, it does become apparent after rescaling: Figure \ref{fig2} illustrates the result of rescaling the numerical solution for $\lambda=1$ by the factor $\length^{-1}$ at each timestep. Since \eqref{GF2} is equivariant under scaling we expect this rescaled numerical solution to be illustrative of the asymptotic behaviour of the actual solution. Our experiments suggest that the asymptotic shape is a round circle irrespective of the value of $\lambda$.

\begin{figure}[h]
    \centering
    \begin{tabular}{cc}
        \includegraphics[width=0.45\textwidth]{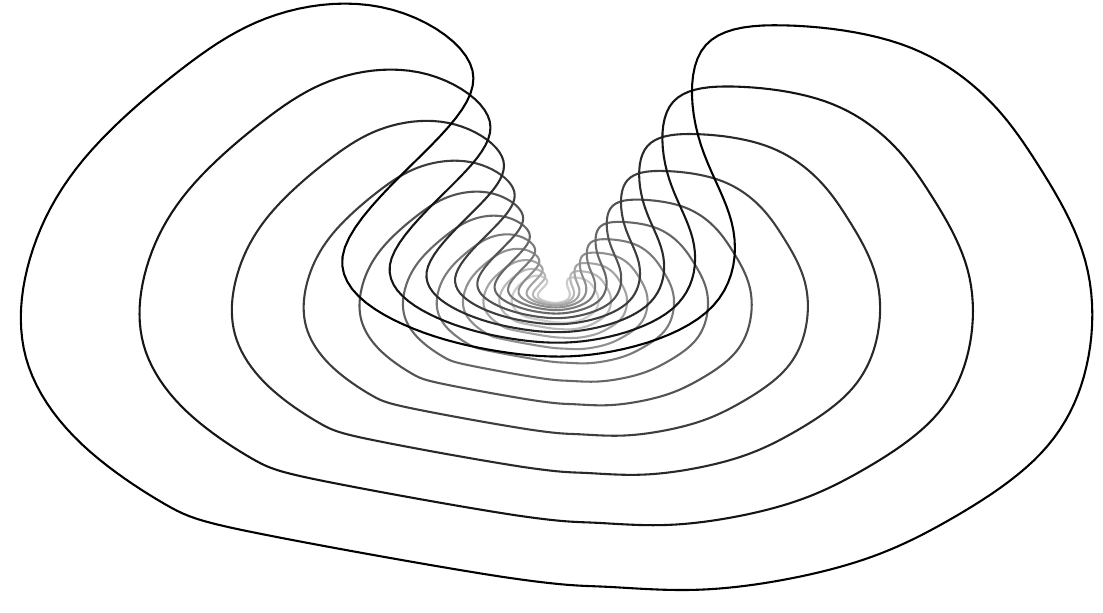} & \includegraphics[width=0.45\textwidth]{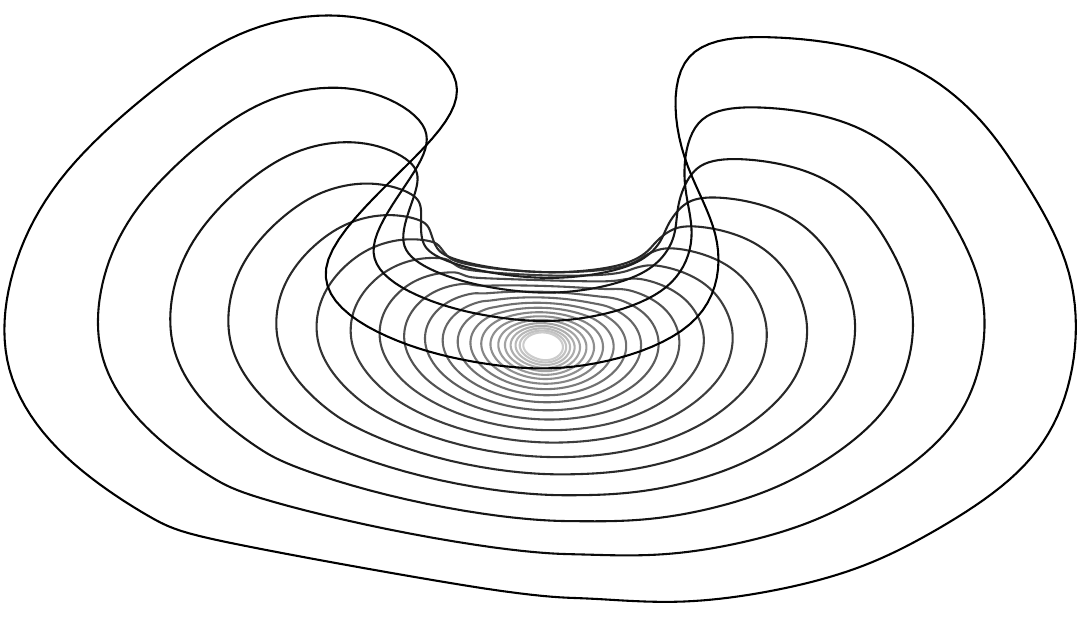} 
    \end{tabular}
    \caption{Numerical solutions to \eqref{GF2} with $\lambda=1$ (left) and $\lambda=0.1$ (right). }
    \label{fig1}
\end{figure}

\begin{figure}[h]
    \centering
        \includegraphics[width=0.45\textwidth]{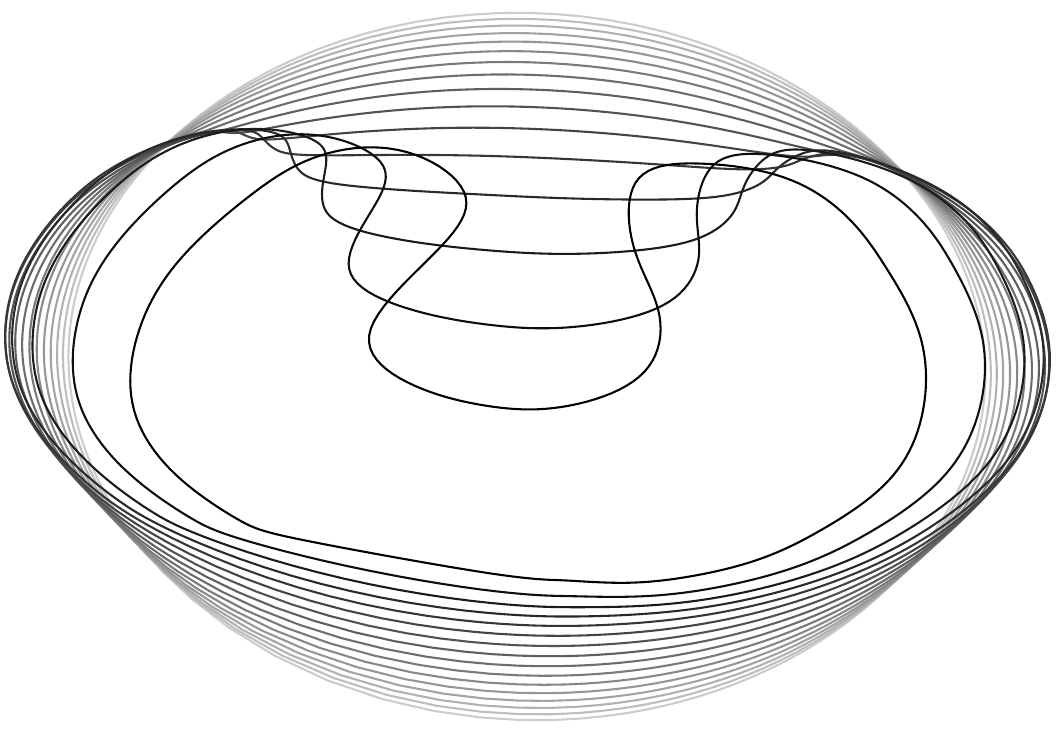}
    \caption{Rescaled numerical solution to \eqref{GF2} with $\lambda=1$.}
    \label{fig2}
\end{figure}

The remainder of the paper is organised as follows. In Section~\ref{sec:formulation} we derive an explicit formula depending on $a$ and $\lambda$ for the gradient of length with respect to the $H^1_{\lambda,a}$ metrics \eqref{metrics}, and record some properties of the Green's function $\G$. Section~\ref{sec:circle} treats the self-similar solution for evolving round circles. Section~\ref{sec:existence} establishes local well-posedness of the gradient flow for each $\lambda>0, a$, by verifying the Lipschitz estimates required to apply the Picard-Lindel\"of theorem. Section~\ref{sec:conv} upgrades this to global existence for $a=2$, proves exponential decay of length and hence convergence to a constant map. These results are then extended to arbitrary $a$ by a time reparametrisation, completing the proof of Theorem~\ref{thm:main1}. Finally, in Section~\ref{sec:convex} we demonstrate that higher differentiability and convexity are preserved along the flow, i.e. Theorem~\ref{thm:main2}.

\section*{Acknowledgements}
The second author acknowledges support from ARC Future Fellowship FT250100880 and ARC Discovery Project DP250101080.
The third author acknowledges support from ARC DECRA DE190100379. 

\section*{Tool and computational resource disclosure}
The numerical simulations and illustrations in Figures \ref{fig1} and \ref{fig2} were generated using the Julia programming language. Claude Opus 4.8 was used to assist in editing and to discuss exposition. This included suggesting references \cite{AmbrosioGigliSavare2008} and \cite{hytonen2016analysis}.

\section{Formulation} \label{sec:formulation}

The flow \eqref{GF} is driven by the gradient of $\length$ with respect to $H^1_{\lambda,a}$, which we now compute in closed form. The main result of this section, Lemma~\ref{gradientlemma}, expresses the gradient as a convolution in normalised arc-length parametrisation against an explicit periodic Green’s function $\G$; we then record several equivalent forms and integral identities for $\G$ which will be used throughout the paper. 

If  $\gamma\in W^{1,1}(\S,\R^2)$ has $|\gamma'(u)|=0$ on a measurable set, then the length functional is not differentiable at $\gamma$ and the $H^1_{\lambda,a}(\gamma)$ metric \eqref{metrics} is not well defined.
However, if we assume that $|\gamma'|>0 \text{ a.e.}$ in $\S$, then the length is G\^ateaux (but not Frech\'et) differentiable, which is sufficient for defining a gradient. Indeed, on the pause set $P:=\{u\in\S:\gamma'(u)=0\}$, the difference quotient of the integrand in $\length$ satisfies
\[
\frac{1}{\varepsilon} (\abs{\gamma'(u)+\varepsilon v'(u)}-\abs{\gamma'(u)})=\sgn(\varepsilon)\abs{v'(u)}, \quad \text{for } u\in P
\]
Since $|\gamma'(u)|$ is otherwise differentiable, the left and right directional derivatives are 
\[
\partial_\varepsilon^\pm \length(\gamma+\varepsilon v)\big|_{\varepsilon=0}
=\int_{\S\setminus P}\frac{\ip{\gamma',v'}}{\abs{\gamma'}}\,du
\pm\int_{P}\abs{v'}\,du.
\]
If the measure of $P$ is non-zero then there are variations for which these left and right derivatives are not equal, but if $|P|=0$ all directional derivatives exist.

We will therefore compute the gradient of $\length$ with respect to $H^1_{\lambda,a}$ on the space
\[
\wimm:=\{\gamma\in W^{1,1}(\S,\R^2): |\gamma'(u)|>0 \text{ a.e.}\}
\]
by first computing the gradient at the constant speed parametrisation:
\[ \hat\gamma(x):= \gamma \circ \rho_\gamma^{-1}(x) ,\quad  \text{where} \quad  \rho_\gamma (u):=\frac{1}{\length(\gamma)}\int_0^u\abs{\gamma'(\tau)}d\tau.\] 
This will then give the gradient at any $\gamma\in\wimm$ using reparametrisation (see \eqref{param}).

It will be convenient to introduce the following abbreviations:
\[
\phat_u:=\frac{\length(\gamma)}{\abs{\gamma'(u)}}\partial_u \quad \text{and}\quad \hat du :=\frac{\abs{\gamma'(u)}}{\length(\gamma)}du,
\]
which satisfy $\hat\partial_u f(u) = \hat{f}'(\rho_\gamma(u))$ and $\int_0^1 f\,\hat{d}u = \int_0^1 \hat{f}(x)\,dx$, where $\hat{f} := f \circ \rho_\gamma^{-1}$.

\begin{lemma}\label{gradientlemma}
The gradient of $\length$ with respect to $H^1_{\lambda,a}$ at $\gamma\in \wimm$, is given by 
\begin{equation}\label{thegradient}
\nabla\length_{\lambda,a}(\gamma)(u)=\frac{\length(\gamma)^{a-2} }{\lambda^2} \left (\gamma(u)+(\gamma \ast_\gamma \mathcal G^\lambda)(u)\right ),
\end{equation}
where $\mathcal{G}^\lambda:\S\to \mathbb R$ is the periodic extension of
\begin{equation}\label{curlyg}
\mathcal{G}^\lambda(x):= -\frac{\cosh\left ( \frac{\abs{x}-1/2}{\lambda}\right )}{2\lambda \sinh\left (\frac{1}{2\lambda}\right)},
\end{equation}
and the $\gamma$-convolution is defined by
\begin{equation}\label{convolution}
(f\ast_\gamma \mathcal G^\lambda)(u):= \int_{0}^1 f(\upsilon)\mathcal{G}^\lambda(\rho_\gamma(\upsilon)-\rho_\gamma(u)) \hat d\upsilon .
\end{equation}
\end{lemma}

\begin{proof}
To find the gradient of length with respect to the metrics \eqref{metrics} we need to solve for $\nabla\length_{\lambda,a}(\gamma)$ in 
\begin{equation}\label{graddef}
D\length(\gamma) v =\ip{\nabla\length_{\lambda,a}(\gamma),v}_{H^1_{\lambda,a}(\gamma)}.
\end{equation}
Arguing as in \cite{schrader2023h} Section 3, the reparametrisation invariance of $\length$ and the $H^1_{\lambda,a}(\gamma)$ metrics ensure that for any reparametrisation $\phi:\S\to \S$ we have $\nabla\length_{\lambda,a}(\gamma\circ \phi)=\nabla\length_{\lambda,a}(\gamma)\circ \phi$. In particular
\begin{equation}\label{param}
\nabla\length_{\lambda,a}(\hat\gamma)(x)=\nabla\length_{\lambda,a}(\gamma)\circ \rho_\gamma^{-1}(x)
\end{equation}
where as above $\hat\gamma:=\gamma\circ \rho_\gamma^{-1}$. We therefore begin by calculating $\nabla\length_{\lambda,a}(\hat\gamma)$.

With $\hat v:=v\circ\rho_\gamma^{-1}$, in the constant speed parametrisation  \eqref{dlength} becomes:
\begin{equation*}
D\length(\gamma) v = \frac{1}{\length(\gamma)}\int_0^1 \ip{\hat \gamma'(x),\hat v'(x)} dx, 
\end{equation*}
the metric is
\[
\ip{v,w}_{H^1_{\lambda,a}(\gamma)}=\frac{1}{\length(\gamma)^{a-1}}\int_0^1 \ip{\hat v(x),\hat w(x)}+\lambda^2 \ip{\hat v'(x),\hat w'(x)} dx,
\]
and then after integrating by parts, in constant-speed parametrisation \eqref{graddef} is equivalent in the sense of distributions to
\[
\lambda^2(\nabla\length_{\lambda,a}(\hat\gamma))''(x)-\nabla\length_{\lambda,a}(\hat\gamma)(x)
=\length^{a-2} \hat\gamma''(x).
\]
We can solve this using the fundamental solution:
\begin{equation}\label{grad}
\nabla\length_{\lambda,a}(\hat\gamma)(x)=\length(\gamma)^{a-2}\,\int_0^{1}  \hat\gamma''(\tau)\mathcal{G}^{\lambda}(\tau-x)\, d\tau
\end{equation}
where $\mathcal G^{\lambda}:[0,1]\to \R $ is the solution to
\begin{equation}\label{delta2}
\left (\lambda^2\partial_x^2-1 \right )\mathcal G^{\lambda}(x)=\delta(x)
\end{equation}
with periodic boundary conditions and the appropriate jump discontinuity in the first derivative.
The solution is given by \eqref{curlyg},
and we use the same notation $\mathcal{G}^\lambda$ for the periodic extension.

Integrating by parts (again in the distributional sense) in \eqref{grad} and using \eqref{delta2} gives
\begin{equation*}
\nabla\length_{\lambda,a}(\hat\gamma)(x)=\frac{\length (\gamma)^{a-2} }{\lambda^2} \left ( \hat \gamma(x)+(\hat \gamma \ast \mathcal{G}^\lambda)(x)\right )
\end{equation*}
and then converting back to an arbitrary parametrisation using \eqref{param}:
\begin{equation}\label{grad2}
\nabla\length_{\lambda,a}(\gamma)(u)=\frac{1}{ \lambda^2}\length (\gamma)^{a-2} \left (\gamma(u)+\int_0^1 \gamma(\upsilon)\mathcal{G}^\lambda\big( \rho_\gamma(\upsilon)-\rho_\gamma(u)\big ) \hat d\upsilon \right ).
\end{equation}
Using \eqref{convolution}, we can write \eqref{grad2} more compactly as \eqref{thegradient}.
\end{proof}

We conclude this section with some additional useful formulae.
With ${\hat\gamma''(\rho_\gamma^{-1}(u))=\hpartial^2_u\gamma(u)}$, the RHS of \eqref{grad} is proportional to the $\gamma$-convolution and we have the following equivalent expression for the gradient:
\begin{equation}\label{Kgrad}
    \nabla\length_{\lambda,a}(\gamma)=\length(\gamma)^{a-2}(\hpartial^2 \gamma \ast_\gamma \mathcal{G}^\lambda).
\end{equation}
Integrating \eqref{delta2} against a constant function and recalling that the second derivative of $\mathcal G^{\lambda} $ is a distributional derivative, we find that
\begin{align}\label{curlygintegral} 
&\int_0^1 \mathcal G^\lambda  (x)dx =-1 = -\int_0^1 \abs{\mathcal G^\lambda (x)} dx 
\end{align}
where the second equality holds because $\mathcal G^\lambda (x)<0$ for all $x\in [0,1]$. For the derivative of $\G$ we calculate 
\begin{equation*}
(\G) '(x)=-\frac{\sinh(\frac{\abs{x}-1/2}{\lambda})}{2\lambda^2 \sinh(\frac{1}{2\lambda})}\sgn(x),
\end{equation*} 
which is positive for $0\leq x\leq 1/2$, and satisfies $(\G) '(x)=- (\G) '(1-x)$. Therefore
\begin{equation}\label{Gprime}
\int_0^1 \big |( \G) '(x)\big |dx=2\int_0^{1/2}(\G) '(x)=\frac{\cosh(\frac{1}{2\lambda})-1}{\lambda \sinh (\frac{1}{2\lambda})}=\tfrac{1}{\lambda}\tanh(\tfrac{1}{4\lambda})
\end{equation}

\section{Evolution of a circle}
\label{sec:circle}

Before turning to the analysis of general solutions to \eqref{GF}, we examine the simplest non-trivial solutions: those with initial data equal to a round circle. 
Circles evolve self-similarly under \eqref{GF}, and exhibit the transition at $a=2$ from finite- to infinite-time extinction that we will also observe in the general case.

With the ansatz $X(u,t)=r(t)(\cos 2\pi u, \sin 2\pi u)$, direct calculation gives $\hat \partial^2 X=-(2\pi)^2 X$ and $X_t=\frac{r'}{r}X$. On the other hand, if $X$ satisfies \eqref{GF} then 
\begin{align*}
 X_t&=-\frac{1}{\lambda^{2}}\length(X)^{a-2}\left(X+X\ast_X\mathcal G^\lambda\right)\\
&=-\frac{1}{\lambda^{2}}(2\pi r)^{a-2}\left (X-\frac{1}{(2\pi)^2}\hat\partial^2 X\ast_X\mathcal G^\lambda \right)\\
&=-\frac{1}{\lambda^{2}}(2\pi r)^{a-2}\left (X+\frac{(2\pi r)^{2-a}}{(2\pi)^2}X_t\right )
\end{align*}
where we have used \eqref{Kgrad}. Rearranging: 
\[ X_t=- \frac{(2\pi)^a}{1+(2\pi\lambda)^2} r^{a-2} X \]
from which it follows that 
\[ r'(t)=- \frac{(2\pi)^a}{1+(2\pi\lambda)^2}r^{a-1}.
\]
Setting $b_{\lambda,a}:=\frac{(2\pi)^a}{1+(2\pi\lambda)^2}$ and $r_0=r(0)$, we find
\[
r(t)=
\begin{cases}
   \left ( r_0^{2-a}-(2-a)b_{\lambda,a}t\right )^{\frac{1}{2-a}} &\quad  a\neq 2\\
   r_0 e^{-b_{\lambda,a}t} & \quad a=2\\
\end{cases}
\]
which decays to zero in finite time if $a<2$ and infinite time if $a\geq 2$.

\section{Existence and uniqueness of solutions}
\label{sec:existence}

From \eqref{grad2} we observe that $\nabla\length_{\lambda,a}(\gamma)$ is well-defined even for those $\gamma\in \wsr$ which are not immersed. 
It is therefore possible to consider \eqref{GF} as an ODE on all of $\wsr$, and to prove the existence and uniqueness of solutions using the Picard-Lindel\"of theorem. 
This requires local boundedness and local Lipschitz estimates for $F_{\lambda,a}$. 
To begin with we target the special case $a=2$:
\begin{equation}\label{GF2}\tag{GF2}
\partial_t X(u,t)=F_{\lambda,2}(X)=
\begin{cases}
-\frac{1}{ \lambda^2} \left (X+X\ast_X\mathcal{G}^\lambda \right ) & \length(X)\neq 0 \\
0 & \length(X)=0
\end{cases}
\end{equation}

From now on we abbreviate $\nabla\length_{\lambda,2}$ to $\nabla\length_{\lambda}$ and $F_{\lambda,2}$ to $F_\lambda$. 

As in Section 2 it will often be convenient to work in constant speed parametrisation; however, having dropped the assumption that $|\gamma'(u)|>0$, invertibility of $\rho_\gamma$ is no longer guaranteed. 
Nevertheless a constant-speed reparametrisation can be defined by adapting a standard approach for arc-length reparametrisation that applies to rectifiable curves (e.g., Lemma 1.1.4 in \cite{AmbrosioGigliSavare2008}). 
We define $\varphi_\gamma:[0,1]\to[0,1]$ by 
\[
\varphi_\gamma(x):=\min\{u\in [0,1]:\rho_\gamma(u)=x\}
\]
so that $\varphi_\gamma$ is a right inverse of $\rho_\gamma$. 
Indeed, on the left we have 
\[ 
\varphi_\gamma(\rho_\gamma(u))=\min\{u_*\in[0,1]:\rho_\gamma(u_*)=\rho_\gamma(u)\}\leq u
\]
and therefore 
\begin{equation*}
\abs{\gamma\circ \varphi_\gamma\circ\rho_\gamma(u)-\gamma(u)}\leq \int_{\varphi_\gamma\circ\rho_\gamma(u)}^u \abs{\gamma'(\tau)}d\tau=\length(\gamma)(\rho_\gamma(u)-\rho_\gamma\circ\varphi_\gamma\circ\rho_\gamma(u))=0.
\end{equation*}
Defining the reparametrisation $\hat\gamma(x):=\gamma\circ\varphi_\gamma(x)$, the above shows that $\gamma=\hat\gamma\circ \rho_\gamma$. If $\gamma$ has pauses (measurable intervals where $|\gamma'(u)|=0$) then $\varphi_\gamma$ will have discontinuities, but $\hat\gamma(x)$ nevertheless turns out to be continuous and differentiable a.e., with $|\hat \gamma'(x)|=\length(\gamma)$ (see Appendix \ref{appendixA}). 

The following lemma collects two bounds on $F_\lambda$ which together control its $W^{1,1}$ norm. 

\begin{lemma}\label{bounds}
For any $\gamma\in W^{1,1}(\S,\R^2)$:
\begin{align}\label{Fest1}
\norm{F_\lambda(\gamma)}_{L^\infty}& \leq \frac{1}{2\lambda^2 }\length(\gamma), \quad \text{and} \\ 
\norm{(F_\lambda(\gamma))'}_{L^1}&\leq \frac{2}{\lambda^2}\length(\gamma).  \label{Fdest}
\end{align}
\end{lemma}
\begin{proof}
Assuming $\length(\gamma)$ is nonzero (otherwise the estimates are trivial), and recalling \eqref{curlygintegral}, we can rewrite $F_\lambda$ as 
\begin{equation}\label{eq:Fmod}
F_\lambda(\gamma)(u)=-\frac{1}{\lambda^2} \int_0^1(\gamma(\upsilon)-\gamma(u))\mathcal{G}^\lambda \big(\rho_\gamma(\upsilon)-\rho_\gamma(u)\big) \hat d\upsilon
\end{equation}
and then
\[
\abs{F_\lambda(\gamma)(u)}\leq \frac{1}{\lambda^2}\frac{\length(\gamma)}{2}\int_0^1 \abs{\mathcal{G}^\lambda \big(\rho_\gamma(\upsilon)-\rho_\gamma(u)\big)} \hat d\upsilon
\]
which, using \eqref{curlygintegral} again, gives \eqref{Fest1}.
For the second estimate, let $\hat{\gamma}(x) := \gamma \circ \varphi_{\gamma}(x)$, then
\[F_\lambda(\hat{\gamma})(x) = -\frac{1}{\lambda^2} \left( \hat{\gamma}(x) + \int_0^1 \hat{\gamma}(\tau)\, \mathcal G^\lambda(\tau - x)\, d\tau \right)\]
and using $\partial_x\mathcal G^\lambda(\tau-x)=-\partial_\tau\mathcal G^\lambda(\tau-x)$ with an integration by parts
\begin{equation}\label{GIP}
(F_\lambda(\hat{\gamma}))'(x) = -\frac{1}{\lambda^2} \left( \hat{\gamma}'(x) + \int_0^1 \hat{\gamma}'(\tau)\, \mathcal G^\lambda(\tau - x)\, d\tau \right)
\end{equation}
Recalling that $|\hat{\gamma}'(x)| = \mathcal{L}(\gamma)$, we estimate $|(F_\lambda(\hat{\gamma}))'(x)| \leq \frac{2}{\lambda^2} \mathcal{L}(\gamma)$. With the substitution $\tau = \rho_\gamma(v)$
\begin{equation}\label{Finvariance}
\begin{aligned} 
F_\lambda(\hat{\gamma})(\rho_\gamma(u)) &= -\frac{1}{\lambda^2} \left( \gamma(u) + \int_0^1 \gamma(v) \mathcal{G}^\lambda(\rho_\gamma(v) - \rho_\gamma(u)) \frac{|\gamma'(v)|}{\mathcal{L}(\gamma)} \, dv \right)
= F_\lambda(\gamma)(u) 
\end{aligned}
\end{equation}
and then
\begin{equation}\label{Fdashest}
|(F_\lambda(\gamma))'(u)| =  |(F_\lambda(\hat{\gamma}))'(\rho_\gamma(u))|\frac{|\gamma'(u)|}{\mathcal{L}(\gamma)} \leq \frac{2}{\lambda^2} |\gamma'(u)|,
\end{equation}
from which we obtain \eqref{Fdest}.
\end{proof}

Below we use $B_\varepsilon(\gamma_0)$ to denote the open $W^{1,1}$-ball centred at $\gamma_0$ with radius $\varepsilon$.
\begin{lemma} \label{lipschitz}
$F_\lambda$ is locally Lipschitz continuous on $W^{1,1}(\S,\R^2)$ away from constant maps. That is, for any $\gamma_0\in W^{1,1}(\S,\R^2)$ with nonzero length there exists $0<\varepsilon< \length(\gamma_0)$ and $c>0$ such that 
\begin{equation}\label{lip}
\norm{F_\lambda(\gamma)-F_\lambda(\beta)}_{W^{1,1}}\leq c \norm{\gamma-\beta}_{W^{1,1}}
\end{equation}
for all $ \beta, \gamma\in B_\varepsilon(\gamma_0)$.
\end{lemma}
\begin{proof}
Let $\sigma_\gamma(u):=\int_0^u \abs{\gamma'(\tau)}d\tau $. For $\beta,\gamma\in W^{1,1}(\S,\R^2)$ we have 
\begin{equation}\label{lip1}
|\sigma_\gamma(u)-\sigma_\beta(u)|\leq \int_0^u\abs{\gamma'(\tau)-\beta'(\tau)}d\tau\leq \norm{\gamma-\beta}_{W^{1,1}}
\end{equation}
showing that $\sigma$ is Lipschitz, and with $u=1$ that $\length$ is Lipschitz. Therefore $\rho_\gamma=\frac{\sigma_\gamma}{\length(\gamma)}$ is Lipschitz away from $\length(\gamma)=0$, being a quotient of Lipschitz functions. Choosing $\varepsilon < \length(\gamma_0)$, so that for all $\gamma\in B_\varepsilon(\gamma_0)$
\begin{equation}\label{lip2}
 \abs{\length(\gamma)-\length(\gamma_0)}\leq \norm{\gamma'-\gamma_0'}_{L^1}\leq \norm{\gamma-\gamma_0}_{W^{1,1}}<\varepsilon   
\end{equation}
and therefore $\length(\gamma)>\length(\gamma_0)-\varepsilon>0$, we have that $\rho$ is Lipschitz on $B_\varepsilon(\gamma_0)$.
From \eqref{GIP} and \eqref{Finvariance}, recalling $\gamma'(u)=\hat\gamma'(\rho_\gamma(u))\frac{\abs{\gamma'(u)}}{\length(\gamma)}$
\begin{equation}\label{fullFdash}
(F_\lambda(\gamma))'(u)=(F_\lambda(\hat \gamma))'(\rho_\gamma (u))\frac{\abs{\gamma'(u)}}{\length(\gamma)}=-\frac{1}{\lambda^2}\left (\gamma'(u)+ \frac{\abs{\gamma'(u)}}{\length(\gamma)}\int_0^1  \gamma'(\upsilon) \mathcal{G}^\lambda (\rho_\gamma(\upsilon)-\rho_\gamma(u)) d \upsilon \right )
\end{equation}
Finally, from \eqref{Finvariance} and \eqref{fullFdash}, using Lipschitz estimates for $\mathcal G^\lambda$, $\rho$ and $\length$ and the usual telescoping sum of products, it follows that 
\begin{align*}
    \norm{F_\lambda (\gamma)-F_\lambda(\beta)}_{L^1} &\leq c \norm{\gamma-\beta}_{W^{1,1}}, \text{ and }\\
    \norm{F_\lambda (\gamma)'-F_\lambda(\beta)'}_{L^1}&\leq c \norm{\gamma-\beta}_{W^{1,1}}
\end{align*}
for any $\gamma,\beta \in B_\varepsilon(\gamma_0)$.
\end{proof}

It follows from Lemma \ref{bounds} and Lemma \ref{lipschitz} that $F_\lambda$ is continuous on all of $\wsr$. 
In order to extend the results of these lemmata to $F_{\lambda,a}$, we use the fact that $F_{\lambda,a}=\length^{a-2}F_\lambda$.

\begin{corollary}\label{Fa-bounds}
Fix $a\in\R$ and $\lambda>0$. For every $\gamma\in W^{1,1}(\S,\R^2)$ with $\length(\gamma)>0$,
\begin{align}\label{Fa-est1}
\norm{F_{\lambda,a}(\gamma)}_{L^\infty}&\leq \frac{1}{2\lambda^2}\length(\gamma)^{a-1},\\
\label{Fa-est2}
\norm{(F_{\lambda,a}(\gamma))'}_{L^1}&\leq \frac{2}{\lambda^2}\length(\gamma)^{a-1}.
\end{align}
Moreover, $F_{\lambda,a}$ is locally Lipschitz continuous on $W^{1,1}(\S,\R^2)$ away from $\{\gamma:\length(\gamma)=0\}$.
\end{corollary}
\begin{proof}
The estimates \eqref{Fa-est1}--\eqref{Fa-est2} are immediate from Lemma~\ref{bounds} and $F_{\lambda,a}=\length^{a-2}F_\lambda$.
For Lipschitz continuity, let $\gamma_0,\varepsilon$ and $\gamma\in B_\varepsilon(\gamma_0)$ be as in Lemma~\ref{lipschitz}. Since $\gamma\mapsto \length(\gamma)$ is Lipschitz on $W^{1,1}$ (by \eqref{lip2}) and $y\mapsto y^{a-2}$ is Lipschitz away from $y=0$, the composition $\gamma\mapsto \length(\gamma)^{a-2}$ is Lipschitz on $B_\varepsilon(\gamma_0)$. The right-hand side of $F_{\lambda,a}=\length^{a-2}F_\lambda$ is therefore locally Lipschitz, being a product of two locally bounded, locally Lipschitz maps into $W^{1,1}$.
\end{proof}

Finally, we are in a position to prove local well-posedness of \eqref{GF} on $\wsr$.

\begin{proposition}\label{existence}
For each $X_0\in W^{1,1}(\S,\R^2)$ with $\length(X_0)>0$, each $a\in\R$ and each $\lambda>0$, there exists $T_0>0$ and a unique
\[
X\in C^1([-T_0,T_0],W^{1,1}(\S,\R^2))
\]
such that $X(0)=X_0$ and $X$ satisfies \eqref{GF}.
\end{proposition}
\begin{proof}
By Corollary~\ref{Fa-bounds} there exist $0<\varepsilon<\length(X_0)$ and $c>0$ such that
\[
\norm{F_{\lambda,a}(\gamma)-F_{\lambda,a}(\beta)}_{W^{1,1}}
\leq c\norm{\gamma-\beta}_{W^{1,1}}
\]
for all $\gamma,\beta\in B_\varepsilon(X_0)$. By \eqref{lip2} every $\gamma\in B_\varepsilon(X_0)$ satisfies
$\length(\gamma)\in I:=[\length(X_0)-\varepsilon,\length(X_0)+\varepsilon]
$
so \eqref{Fa-est1}--\eqref{Fa-est2} give
\[
\norm{F_{\lambda,a}(\gamma)}_{W^{1,1}}
\leq \norm{F_{\lambda,a}(\gamma)}_{L^\infty}+\norm{(F_{\lambda,a}(\gamma))'}_{L^1}\leq \frac{5}{2\lambda^2} \length(\gamma)^{a-1}
\leq \frac{5}{2\lambda^2}\sup_{r\in I}r^{a-1}<\infty
\]
for all $\gamma\in B_\varepsilon(X_0)$.
Hence $F_{\lambda,a}$ is locally Lipschitz and locally bounded on $W^{1,1}$ near $X_0$, and the Picard--Lindel\"of theorem (see e.g.\ \cite{Zeidler:1986aa}, Theorem 3.A) yields $T_0>0$ and a unique solution
\[
X\in C^1([-T_0,T_0],W^{1,1}(\S,\R^2))
\]
to \eqref{GF} with $X(0)=X_0$.
\end{proof}

\section{Global existence and convergence}
\label{sec:conv}

We now turn to the long-time behaviour of solutions. Beginning again with the case $a=2$ we demonstrate: global existence and uniqueness of the stationary constant solutions (Lemma~\ref{wellposedconstants}), monotonicity of length and sup-norm (Lemma~\ref{apriori}), a lower bound on length precluding finite-time collapse (Proposition~\ref{globalexistence}), exponential length decay (Lemma~\ref{lengthdecay}), and convergence to constant maps (Proposition~\ref{longtimebehaviour}). Having established global existence and convergence with $a=2$ we show in Proposition \ref{timescaling} that solutions to \eqref{GF2} and \eqref{GF} differ only by a time reparametrisation, and that this reparametrisation maps to a finite time interval when $a<2$, completing the proof of Theorem \ref{thm:main1}.

\begin{remark}
Proposition \ref{existence} gives solutions to \eqref{GF} as curves $t\mapsto X(t)$ taking values in $\wsr$, but in \eqref{GF} $X$ is a function of two variables $(u,t)$. Reconciling the derivatives of $X$ from these two viewpoints requires the isometric isomorphism 
\[  \iota:L^1([0,T],L^1(\S,\R^2))\cong L^1(\S\times [0,T],\R^2)\] defined by $\iota X(u,t)=X(t)(u)$ (see e.g. \cite{hytonen2016analysis} Proposition 1.2.24). As usual, we identify elements of $\wsr$ with their absolutely continuous representatives, so that given $X\in C^1([0,T],\wsr)$, $\partial_u \iota X(u,t)$ exists a.e. in $\S\times [0,T]$ and is equal to $\iota\partial X(u,t)$ in $L^1(\S\times [0,T],\R^2)$, where $\partial : \wsr \to L^1(\S,\R)$ denotes the weak spatial derivative. Moreover, writing $D:C^1([0,T],\wsr)\to C^0([0,T],\wsr)$ for the time derivative, we also have that 
\[
\iota D\partial X=\partial_t\partial_u\iota X=\partial_u\partial_t\iota X=\iota \partial DX \quad \text{in } L^1(\S\times [0,T],\R^2).
\]
From now on we omit $\iota$ and make use of the above identifications without mention; this is the sense in which the expressions below involving $\partial_u X(u,t)$ should be interpreted.

\end{remark}


\begin{lemma}\label{wellposedconstants}
For each $X_0\in W^{1,1}(\S,\R^2)$ with $\length(X_0)=0$, each $a\in\R$ and each $\lambda>0$, the unique $X\in C^1([0,\infty),W^{1,1}(\S,\R^2))$ satisfying \eqref{GF2} and $X(\cdot,0)= X_0$ is the stationary solution $X(\cdot, t)\equiv X_0$.
\end{lemma}
\begin{proof}
If $X(\cdot,t)\equiv X_0$ then $\partial_t X=0$ and $\length(X(\cdot,t))=0$ for all $t$, so $X$ satisfies \eqref{GF}. To see that this solution is unique, let $X\in C^1([0,\infty),W^{1,1}(\S,\R^2))$ be any solution to \eqref{GF} with $X(0)=X_0$ constant, and set 
$\ell(t):=\length(X(\cdot,t))=\norm{X_u(\cdot,t)}_{L^1}$.
Since $t\mapsto X_u(\cdot,t)$ is $C^1$ as an $L^1$-valued map 
\[
\abs{\ell(t)-\ell(0)}\leq \norm{X_u(\cdot,t)-X_u(\cdot, 0)}_{L^1}\leq \int_0^t\norm{\partial_\tau X_u(\cdot, \tau)}_{L^1}d\tau \leq \int_0^t \frac{2}{\lambda^2}\ell(\tau)
\]
by Lemma~\ref{bounds}. As $\ell(0)=0$, Gr\"onwall's inequality gives $\ell(t)\equiv 0$, i.e. $X(\cdot,t)$ is constant in $u$ for every $t$. Then \eqref{GF2} implies $\partial_tX=0$ and therefore $X(\cdot,t)\equiv X_0$.
\end{proof}

\begin{lemma}\label{apriori}
If $X\in C^1([-T_0,T_0],W^{1,1}(\S,\R^2))$ is a solution to \eqref{GF2}, then $\norm{X(\cdot, t)}_{L^\infty}$ and $\length(X(\cdot,t))$ are non-increasing. 
\end{lemma}
\begin{proof} 
With $V(t):= \norm{X(\cdot,t)}_{L^\infty}$:
\begin{align*}
|V(t_2)-V(t_1)|&\leq \sup_{u\in \S}\left | X(u,t_2)-X(u,t_1)\right |=\sup_{u\in \S}\left | \int_{t_1}^{t_2} \partial_t X(u,\tau) d\tau \right | \\
&\leq \int_{t_1}^{t_2}\sup_{u\in \S}|\partial_t X(u,\tau)|d\tau 
\end{align*}
Then since $X$ is $C^1$, $V$ is Lipschitz and therefore differentiable almost everywhere. 
By Danskin's theorem \cite{danskin1966theory}, where it exists $V'(t)$ satisfies
\[ V'(t) = \max_{u^*\in U^*} \partial_t |X(u^*,t)|, \quad \text{where}\quad  U^*(t):=\{ u\in \S: |X(u,t)|=V(t)\}.
\]
For any $u\in \S$, using \eqref{curlygintegral}
\begin{align*}
\partial_t |X(u,t)|^2 & = - \frac{2}{\lambda^2} \langle X(u,t), X(u,t)+ (X\ast_X \mathcal G^\lambda)(u,t) \rangle \\
& = -\frac{2}{\lambda^2} |X(u,t)|^2-\frac{2}{\lambda^2}\Big \langle X(u,t),\int X(\upsilon,t)\mathcal G^\lambda(\rho_X(\upsilon)-\rho_X(u)) \hat d \upsilon \Big \rangle  \\
& \leq  -\frac{2}{\lambda^2} |X(u,t)|^2 +\frac{2}{\lambda^2}|X(u,t)|\norm{X(t)}_{L^\infty}
\end{align*}
and then for all $u^*\in U^*(t)$ we have $\partial_t |X(u^*,t)|^2\leq 0$.  Then assuming $|X(u^*,t)|>0$, we have
\[V'(t)=\max_{u^*\in U^*} \partial_t |X(u^*,t)| \leq 0 \,\, \text{a.e.} \]
On the other hand, if there is a $t^*$ such that $|X(u^*,t^*)|=0$, then $|X(u,t^*)|=0$ for all $u\in \S$, $\length(X(\cdot,t^*))=0$, and by Lemma \ref{wellposedconstants} $\norm{X(\cdot, t)}_{L^\infty}=0$ and $\length(X(\cdot,t))=0$ for all $t\geq t^*$.

Recall that for a.e. $u\in S$, the map $t\mapsto X_u(u,t)$ is absolutely continuous,  then so is $t\mapsto |X_u(u,t)|^2$, and $\partial_t|X_u(u,t)|^2=2\langle X_u(u,t),X_{ut}(u,t)\rangle $ a.e. in $t$. Hence, while $\length(X(\cdot,t))>0$, recalling \eqref{GIP} and \eqref{Finvariance} we have 
\begin{align*}
\partial_t |X_u(u,t)|^2
& = -\frac{2}{\lambda^2} |X_u(u,t)|^2 -\frac{2}{\lambda^2}\frac{\abs{X_u(u,t)}}{\length(X)}\Big \langle X_u(u,t),\int \hat X_\tau(\tau,t)\mathcal G^\lambda(\tau-\rho_X(u)) d \tau \Big \rangle\\
& \leq -\frac{2}{\lambda^2} |X_u(u,t)|^2+\frac{2}{\lambda^2}\frac{\abs{X_u(u,t)}^2}{\length(X)}{\Vert\hat X_\tau(\cdot, t)\Vert}_{L^\infty}
\end{align*}
where $\hat X(\tau,t):=X(\varphi_X(\tau),t)$ satisfies $|\hat X_\tau(\tau, t)|=\length(X)$, and therefore $\partial_t |X_u(u,t)|^2 \leq 0 $. Then $|X_u(u,t_2)|^2\leq |X_u(u,t_1)|^2$ whenever $0\leq t_1\leq t_2$, hence $|X_u(u,t_2)|\leq |X_u(u,t_1)|$ and also $\length(X(t_2))\leq \length(X(t_1))$.
\end{proof}

The previous lemma establishes that solutions remain bounded in $\wsr$. To prove global existence it remains to rule out the possibility that length decays to zero in finite time, which would push the solution out of the region where $F_\lambda$ is Lipschitz.

\begin{proposition}\label{globalexistence}
For every $X_0\in W^{1,1}(\S,\R^2)$ there exists a unique global solution
$X\in C^1([0,\infty),W^{1,1}(\S,\R^2))$ to \eqref{GF2} with $X(\cdot,0)=X_0$.
\end{proposition}
\begin{proof}

By Lemma \ref{wellposedconstants} it suffices to consider $\length(X_0)>0$.
Let
$
X\in C^1([0,T_{\max}),W^{1,1}(\S,\R^2))
$
be the positive-time-maximal solution given by Proposition~\ref{existence}.
By Lemma~\ref{bounds} and \ref{apriori},
\[
\norm{\partial_t X(\cdot,t)}_{W^{1,1}}
=\norm{F_\lambda(X(\cdot,t))}_{W^{1,1}}
\le \frac{5}{2\lambda^2}\length(X(\cdot,0))
\]
and so for any $t_1\leq t_2\leq T_{\max}$ we have
\[
\norm{X(\cdot,t_2)-X(\cdot,t_1)}_{W^{1,1}}\leq \Big\Vert \int_{t_1}^{t_2}\partial_t X(\cdot,t)dt\Big\Vert_{W^{1,1}}\leq |t_2-t_1|\frac{5}{2\lambda^2}\length(X(\cdot,0)).
\]
Therefore if $T_{\max}<\infty$ there exists $X_*\in W^{1,1}(\S,\R^2)$ such that
$X(\cdot,t)\to X_*$
in $W^{1,1}(\S,\R^2)$
as $t\uparrow T_{\max}.$

To apply Proposition~\ref{existence} with $X_*$ as initial value we require $\length(X_*)>0$.
Set $\ell(t):=\length(X(\cdot,t))=\norm{X'(\cdot,t)}_{L^1}$. 
Since $t\mapsto X'(\cdot,t)$ is $C^1$ as an $L^1$-valued map, we have
\[
|\ell(t_2)-\ell(t_1)|\leq \norm{X_u(\cdot,t_2)-X_u(\cdot,t_1)}_{L^1}\leq \int_{t_1}^{t_2}\norm{X_{ut}(\cdot, \tau)}_{L^1}d\tau
\]
hence $\ell$ is absolutely continuous and for a.e.\ $t<T_{\max}$,
\[
\ell'(t)\ge -\norm{X_{ut}(\cdot,t)}_{L^1}
=-\norm{F_\lambda(X(\cdot,t))'}_{L^1}
\ge -\frac{2}{\lambda^2}\ell(t),
\]
where we have used \eqref{Fa-est2}. Gr\"onwall's inequality then yields
\begin{equation}\label{lengthlowerbound}
\ell(t)\ge \ell(0)e^{-2t/\lambda^2}>0
\qquad\text{for all }t<T_{\max}.
\end{equation}
and by continuity of $\length$,
\[
\length(X_*)=\lim_{t\uparrow T_{\max}}\ell(t)\ge \ell(0)e^{-2T_{\max}/\lambda^2}>0.
\]
Now Proposition~\ref{existence} with $X_*$ as initial value gives a unique continuation of $X$ beyond time $T_{\max}$, contradicting the assumption of maximality. Therefore $T_{\max}=\infty$.
\end{proof}

If $X$ is immersed then $\length(X)$ is differentiable and we can do better than just monotonicity, but first we need to show that if the initial data is immersed then this property is preserved along the flow. 

\begin{lemma}
\label{PRimmersed}
    If $X_0\in \wimm$ then the solution $X$ to \eqref{GF2} with $X(0)=X_0$ remains immersed, i.e.  $X\in C^1([0,T_0],\wimm)$. 
\end{lemma}
\begin{proof}
From \eqref{GF2}
\[
\frac{1}{2} \frac{d}{dt}|X_u|^2=\ip{X_u,X_{tu}}=-\frac{1}{\lambda^2}|X_u|^2-\frac{1}{\lambda^2}\ip{X_u,\partial_u(X\ast_X\mathcal{G}^\lambda)},
\]
and arguing as in the proof of Lemma \ref{bounds} we estimate $\big|\partial_u(X\ast_X\mathcal G^\lambda)(u,t)\big|\leq |X_u(u,t)| $. 
Hence
\[
\frac{d}{dt}|X_u(u,t)|^2\geq-\frac{4}{\lambda^2}|X_u(u,t)|^2,
\]
from which it follows that
\begin{equation}\label{immersionbound}
 |X_u(u,t)|^2\geq |X_u(u,0)|^2e^{-4t/\lambda^2}.   
\end{equation}
Then because $X_0$ is immersed, $|X_u(u,t)|>0$ for all $t$.
\end{proof}

\begin{lemma} \label{lengthdecay}
Let $X$ be a solution to \eqref{GF2} with $X_0$ immersed. Abbreviate $\mathcal L (t)=\mathcal L(X(\cdot,t))$ and $\mathcal L_0=\mathcal L(0)$, then
\begin{align}\label{eq:ldecay}
\length(t)\leq \length_0 e^{-4t/(1+8\lambda^2)}
\,.
\end{align}
\end{lemma}
\begin{proof} By Lemma \ref{PRimmersed}, $X(\cdot,t)$ is immersed and therefore $\rho_{X(\cdot,t)}$ is invertible and $\length$ is differentiable at $X$.
Write $\hat X(x,t):=X(\rho_{X(\cdot,t)}^{-1}(x),t)$.
From \eqref{dlength}, and using \eqref{Finvariance}:
\[
\begin{aligned}
\mathcal{L}^{\prime}(t)  =D \mathcal{L}(X) \cdot \partial_t X
&=\frac{1}{\length(X)}\int_0^1\left\langle\hat{\partial}_u X, \partial_u \partial_t X\right\rangle du\\
& =-\frac{1}{\lambda^2\length(X)} \int_0^{1}\left\langle\partial_x \hat X, \partial_x \left(\hat X+\hat X * \mathcal G^\lambda \right)\right\rangle dx \\
& =-\frac{\mathcal{L}}{\lambda^2}-\frac{1}{\lambda^2\length(X)} \int_0^{1} \langle \partial_x \hat X,\partial_x (\hat{X}\ast \mathcal G^\lambda)\rangle dx .
\end{aligned}
\]
We define $\bar{X}:=\int_0^1 X du$, and note that ${\partial}_x(\hat X-\bar{X})=\partial_x \hat X$. Since
\[
\begin{aligned}
\partial_x^2(\hat X \ast \mathcal G^\lambda) & =\frac{1}{\lambda^2} \hat X\ast (\delta+\mathcal G^\lambda)  =- F_\lambda(\hat X),
\end{aligned}
\]
after an integration by parts and another application of \eqref{Finvariance} we obtain
\[
\begin{aligned}
\mathcal{L}^{\prime}(t) 
& =-\frac{\mathcal{L}}{\lambda^2}-\frac{1}{\lambda^2\length(X)} \int_0^1\left\langle X-\bar{X}, \partial_t X\right\rangle \hat{d}u .
\end{aligned}
\]
Then using $\abs{X-\bar X}\abs{\partial_tX}\leq \varepsilon\abs{X-\bar X}^2+\frac{1}{4\varepsilon}\abs{\partial_tX}^2$, and also    $\norm{X-\bar X}_{L^\infty}\leq \length/2$
\begin{align*}
\length'(t)
& \leq -\frac{\length}{\lambda^{2}}+\frac{\varepsilon \length }{4\lambda^{2}}+ \frac{1}{4\varepsilon \lambda^2\length^2}\int_0^1 \abs{\partial_t X}^2 \abs{\partial_u X} du.
\end{align*}
Recalling \eqref{metrics}, since $\length'(t)=-\norm{\partial_t X}_{H^1_\lambda (X)}^2$, it now follows that 
\[
\length'(t)
 \leq -\frac{\length}{\lambda^{2}}+\frac{\varepsilon \length }{4\lambda^{2}}- \frac{1}{4\varepsilon \lambda^2}\length'(t),
\]
hence
\[
\length'(t)\leq -\frac{1-\varepsilon/4}{1/(4\varepsilon)+\lambda^2}\length \]
and setting $\varepsilon=2$ and integrating gives the stated result.  
\end{proof}


The following lemma shows that solutions to \eqref{GF2} are equivariant even under reparametrisations that are not necessarily invertible. This will be used to extend the length decay estimate \eqref{lengthdecay}, via constant speed reparametrisation, to the general case where $X_0$ is not necessarily immersed.

\begin{lemma}\label{reparam}
Let $\phi : [0,1] \to [0,1]$ be absolutely continuous, non-decreasing and surjective.
If $\tilde{X}(x,t)$ is the solution to (GF2) with $\tilde{X}(\cdot, 0) = \tilde{X}_0$, then
$X(u,t) := \tilde{X}(\phi(u), t)$ is the solution with $X(\cdot, 0) = \tilde{X}_0 \circ \phi$.
\end{lemma}

\begin{proof}
Note that
$$\sigma_X(u) = \int_0^u \left| \partial_\upsilon \tilde{X}(\phi(\upsilon ), t) \right| d\upsilon = \int_0^u \left| \partial_1 \tilde{X}(\phi(\upsilon ), t) \right| |\phi'(\upsilon )|\, d\upsilon $$
then since $\phi'(u) \geq 0$, substituting $\tau = \phi(\upsilon )$ gives $\sigma_X(u) = \sigma_{\tilde{X}}(\phi(u))$.
Hence $\mathcal{L}(X) = \mathcal{L}(\tilde{X})$ and $\rho_X(u) = \rho_{\tilde{X}}(\phi(u))$, and then since $\tilde X$ satisfies \eqref{GF2},
\begin{align*}
X_t(u,t) &= \tilde{X}_t(\phi(u), t) \\
&= -\frac{1}{\lambda^2} \left( \tilde{X}(\phi(u), t) + \int_0^1 \tilde{X}(\phi(v), t)\, \G\left( \rho_{\tilde{X}}(\phi(v)) - \rho_{\tilde{X}}(\phi(u)) \right) d \rho_{\tilde{X}}(\phi(v)) \right) \\
&= -\frac{1}{\lambda^2} \left( X(u,t) + \int_0^1 X(v,t)\, \G\left( \rho_X(v) - \rho_X(u) \right) d \rho_X(v) \right) \\
&= F_\lambda (X)(u,t)
\end{align*}
shows that $X$ also satisfies \eqref{GF2}, with $X(u,0)=\tilde X(\phi(u),0)$.
\end{proof}

\begin{corollary}\label{lengthdecayall}
    The estimate \eqref{eq:ldecay} extends to any solution $X\in C^1([0,\infty),\wsr)$ to \eqref{GF2}. 
\end{corollary}
\begin{proof}
If $\length(X_0)=0$ then by Lemma \ref{wellposedconstants} \eqref{eq:ldecay} holds trivially.
Given $X_0\in \wsr$ with $\length(X_0)>0$, if we apply Lemma \ref{reparam} with $\phi=\rho_{X_0}$ and $\tilde X_0=\hat X_0=X_0\circ \varphi_{X_0}$, then we have that $X(u,t)=\tilde X(\rho_{X_0}(u),t)$ is the solution with $X(\cdot,0)=X_0$. It then follows that $\length(X(\cdot, t))=\length(\tilde X(\cdot,t)).$ Since $\hat X_0$ is immersed $\length(\tilde X(\cdot,t))$ satisfies \eqref{eq:ldecay}, and then so does $\length(X(\cdot, t))$.
\end{proof}

Having extended the length decay estimate to arbitrary initial data, we now use it to prove that solutions converge to constant maps as $t\to \infty$.

\begin{proposition}\label{longtimebehaviour}
Let 
$X\in C^1([0,\infty),W^{1,1}(\S,\R^2))$ be the solution to to \eqref{GF2} with $X(\cdot,0)=X_0\in W^{1,1}(\S,\R^2)$. Then as $t\to \infty$,
$
X(\cdot,t)
$
converges in $W^{1,1}(\S,\R^2)$, and hence uniformly, to a constant map $u\mapsto \x_\infty \in \R^2$ with ${\abs{\x_\infty}\leq \norm{X_0}_{L^\infty}}$. 
\end{proposition}
\begin{proof}
By Lemma \ref{bounds} and Corollary \ref{lengthdecayall}
\[
\norm{X(\cdot,t_2)-X(\cdot,t_1)}_{W^{1,1}}\leq \Big\Vert \int_{t_1}^{t_2}\partial_t X(\cdot,t)dt\Big\Vert_{W^{1,1}}\leq \frac{5}{2\lambda^2}\length(X(\cdot,0))\int_{t_1}^{t_2}e^{-4t/(1+8\lambda^2)}dt.
\]
Hence $X(\cdot,t_j)$ is Cauchy in $W^{1,1}$ for any sequence of times $t_j\to\infty$, and there exists $X_\infty$ such that
$X(\cdot,t)\to X_\infty$
in $W^{1,1}(\S,\R^2)$
as $t\to\infty.$
By continuity of $\length$ and \eqref{eq:ldecay},
\[
\length(X_\infty)=\lim_{t\to\infty}\length(X(\cdot,t))=0,
\]
and therefore $X_\infty$ is a constant map, say $X_\infty(u)\equiv \x_\infty$. Since $W^{1,1}(\S,\R^2)$ embeds continuously into $C^0(\S,\R^2)$, the convergence is also uniform. Finally, Lemma~\ref{apriori} gives
$
\norm{X(\cdot,t)}_{L^\infty}\le \norm{X_0}_{L^\infty}
$
for all $t\ge0$, and passing to the uniform limit yields
$|\x_\infty|\le \norm{X_0}_{L^\infty}.$

\end{proof}

Since $F_{\lambda,a}=\length^{a-2}F_\lambda$, a time reparametrisation with derivative $\length^{2-a}$  allows us to set up a correspondence between solutions to \eqref{GF} and \eqref{GF2}, and hence extend the above convergence results to the general case. 

\begin{proposition}\label{timescaling}
Fix $a\in\R$, $\lambda>0$ and $X_0\in W^{1,1}(\S,\R^2)$ with $\length(X_0)>0$. The maximal solutions $X\in C^1([0,T_a),\wsr)$ and $Y\in C^1([0,\infty),\wsr)$ to \eqref{GF} and \eqref{GF2} with $X(\cdot,0)=Y(\cdot,0)=X_0$ differ by a time reparametrisation:
\[
Y(\cdot,t)=X(\cdot,\theta(t)),
\]
where $\theta:[0,\infty)\to[0,T_a)$ is defined by $\theta(t):=\int_0^t \length(Y(\cdot, \tau))^{2-a}\,d\tau$.
\end{proposition}
\begin{proof}
Set $\ell(t):=\length(Y(\cdot,t))$
and define
\[
\theta(t):=\int_0^t \ell(\tau)^{2-a}\,d\tau,\qquad t\geq 0,
\]
then $\theta\in C^1([0,\infty),[0,T))$ with $\theta(0)=0$. By \eqref{lengthlowerbound} $\ell(t)>0$ and then $\theta'(t)=\ell(t)^{2-a}>0$, hence $\theta$ is strictly increasing. Recalling \eqref{lengthdecay}, 
$\ell(t)\leq \ell_0e^{-\beta t}$ where $ \ell_0:=\ell(0)$ and $ \beta:=4/(1+8\lambda^2)$, and it follows that $\theta$ maps onto $[0,T)$ where 
\[
T:=\int_0^\infty \ell(s)^{2-a}\,ds \,
\begin{cases}
=\infty & \text{if } a\geq 2,\\
\leq \ell_0^{2-a}/(\beta(2-a))<\infty & \text{if }a<2.
\end{cases}
\]

Set $\bar X(u,t):=Y(u,\theta^{-1}(t))$ for $(u,t)\in\S\times[0,T)$, then since $Y$ satisfies \eqref{GF2}
\[
\partial_t \bar X(\cdot,t)
=\partial_2 Y(\cdot,\theta^{-1}(t))\,(\theta^{-1})'(t)
=F_\lambda(\bar X(\cdot,t))\frac{1}{\ell(\theta^{-1}(t))^{2-a}}
=\length(\bar X(\cdot,t))^{a-2}F_\lambda(\bar X(\cdot,t)),
\]
i.e. $\bar X$ satisfies \eqref{GF} on $[0,T)$ with $\bar X(\cdot,0)=Y(\cdot,0)=X_0$. By uniqueness of solutions to \eqref{GF} (Proposition~\ref{existence} together with the standard continuation argument), $\bar X$ and the maximal solution $X$ agree on $[0,\min(T,T_a))$. If $a\geq 2$, $T=\infty$ and maximality of $X$ forces $T_a=\infty$. If $a<2$, $\length(\bar X(\cdot,t))=\ell(\theta^{-1}(t))\to 0$ as $t\to T^-$, so $\bar X$ cannot be extended as a solution to \eqref{GF} with positive length; the same applies to $X$, giving $T_a=T$. In either case $T=T_a$ and $\bar X=X$ on $[0,T_a)$, and so 
$Y(\cdot,t)=X(\cdot,\theta(t))$ for all $t\geq 0.$
\end{proof}

\begin{remark}
For $a<2$, the bound on the extinction time $T$ has the same dependence on $\ell_0:=\length(X_0)$ and $(2-a)$ as the exact extinction time of the round circle solution computed in Section~\ref{sec:circle}: 
\[
t_{\rm ext}=\frac{(1+(2\pi\lambda)^2)\,\ell_0^{2-a}}{(2-a)\,(2\pi)^{2}}.
\]
Hence the powers $\ell_0^{2-a}$ and $(2-a)^{-1}$ are sharp.
\end{remark}

Theorem \ref{thm:main1} now follows by combining Proposition~\ref{longtimebehaviour} and Proposition~\ref{timescaling}.

\section{Preservation of convexity}
\label{sec:convex}

In this section we show that the flow preserves strict convexity, 
provided the initial datum is $C^2$ and strictly convex. We first prove that if the initial curve is $C^1$ and evolves by \eqref{GF2} then this regularity persists. We then upgrade this to preservation of $C^2$ regularity, and derive an evolution equation for curvature $k$ which, via a Gr\"onwall inequality, ensures that positivity of $k$ is also preserved along the flow. 

\begin{lemma}\label{C1Lipschitz}
For each $X_0\in C^{1}(\S,\R^2)$ there exists a unique $X\in C^1([0,\infty),C^{1}(\S,\R^2))$ such that $X(0)=X_0$ and $X$ satisfies \eqref{GF2}. 
\end{lemma}
\begin{proof}
From \eqref{Fest1} and \eqref{Fdashest} we have that $F_\lambda$ is locally bounded in $C^1$. The $W^{1,1}$ norm in \eqref{lip1} can be replaced by a $C^1$ norm, likewise the subsequent Lipschitz estimates in the proof of Lemma \ref{lipschitz}; we conclude that $F_\lambda$ is also locally Lipschitz in $C^1$. Short-time existence and uniqueness in $C^1$ then follow from the Picard-Lindel\"of theorem as in Proposition \ref{existence}.

To extend the local $C^1$ solution globally it suffices (by the standard continuation criterion for ODEs in Banach spaces, cf.\ \cite{Zeidler:1986aa}) to exclude blow-up of the $C^1$ norm in finite time.
The $L^\infty$-bound in Lemma~\ref{apriori} controls $\|X(\cdot,t)\|_{L^\infty}$, so here we need only estimate $\|X_u(\cdot,t)\|_{L^\infty}$.
By \eqref{Fdashest} we have $|\partial_t X_u(u,t)|\le \frac{2}{\lambda^2}|X_u(u,t)|$, and then Gr\"onwall's inequality yields
\begin{equation}\label{C1bound}
\|X_u(\cdot,t)\|_{L^\infty}\le \|X_u(\cdot,0)\|_{L^\infty}\,e^{2t/\lambda^2}\qquad\text{for all }t\ge0.
\end{equation}
In particular, $\sup_{t\in[0,T]}\|X(\cdot,t)\|_{C^1}<\infty$ for each $T<\infty$, so the maximal $C^1$ solution extends to all $t\ge0$.
\end{proof}

\begin{lemma}
For each immersion $X_0\in C^{2}(\S,\R^2)$ there exists a unique $X\in C^1([0,\infty),C^{2}(\S,\R^2))$ such that $X(0)=X_0$ and $X$ satisfies \eqref{GF2}. 
\end{lemma}
\begin{proof}
Differentiating \eqref{fullFdash}:
\begin{align}
-\lambda^2 F_\lambda(\gamma)''(u)
\label{fullFdashdash}
&=  \begin{multlined}[t]\gamma''(u)+\frac{\ip{\gamma''(u),\gamma'(u)}}{\length(\gamma)\abs{\gamma'(u)}} \int_0^1  \phat \gamma(\upsilon) \mathcal{G}^\lambda (\rho_\gamma(\upsilon)-\rho_\gamma(u)) \hat d \upsilon \\
-\frac{\abs{\gamma'(u)}^2}{\length(\gamma)^2} \int_0^1  \phat \gamma(\upsilon)  (\G)'\big (\rho_\gamma(\upsilon)-\rho_\gamma(u)\big) \hat d \upsilon ,
\end{multlined}
\end{align}
and using \eqref{Gprime}, we estimate
\begin{equation}\label{C2bound}
\norm{F_\lambda(\gamma)''}_{L^\infty} \leq \frac{2}{\lambda^2}\norm{\gamma''}_{L^\infty}+\frac{\tanh(\tfrac{1}{4\lambda})}{\lambda^3 \length(\gamma)}\norm{\gamma'}^2_{L^\infty}.
\end{equation}
Assuming $\gamma_0$ is an immersion, we let $\varepsilon <\frac{1}{2}\min_{u\in\S}|\gamma_0'(u)|=:m_0$ so that, by a similar argument to \eqref{lip2}, $\min_{u\in\S}|\gamma'(u)|\geq m_0>0$ for all $\gamma$ in the open $C^2$-ball $B_\varepsilon^{C^2}(\gamma_0)$. Then $1/|\gamma'|$, $\rho$ and $1/\length$ are Lipschitz on $B_\varepsilon^{C^2}(\gamma_0)$. After integrating by parts in the third summand of \eqref{fullFdashdash}, then expressing it as a suitable telescoping sum of products and using term-by-term Lipschitz estimates we have
\[
\norm{F_\lambda (\gamma)''-F_\lambda(\beta)''}_{L^\infty} \leq c \norm{\gamma-\beta}_{C^2}
\]
for any $\gamma,\beta \in B_\varepsilon^{C^2}(\gamma_0)$, where $c$ is a constant. Combining the above inequality with the $C^1$ Lipschitz estimate explained in the proof of Lemma \ref{C1Lipschitz} yields the $C^2$ local-Lipschitz property required by the Picard--Lindel\"of theorem. Moreover, the $C^1$ bound from Lemma \ref{C1Lipschitz} combines with \eqref{C2bound} to give a $C^2$ bound, and so \eqref{GF2} is well-posed as an evolution in $C^2(\S,\R^2)$. 

As in Lemma \ref{C1Lipschitz}, in order to extend the solution $X(\cdot, t)$ indefinitely it suffices to exclude blow-up of the $C^2$-norm in finite time. Indeed by Lemma \ref{C1Lipschitz}, $\norm{X(\cdot,t)}_{C^1}$ is controlled, so it remains to bound $\norm{X_{uu}(\cdot,t)}_{L^\infty}$. Suppose $X\in C^1([0,T],C^{2}(\S,\R^2))$ is a solution and denote $\ell(t)=\length(X(\cdot, t))$, then \eqref{C2bound}, \eqref{lengthlowerbound} and \eqref{C1bound} yield
\begin{align*}
\norm{\partial_t X_{uu}(\cdot, t)}_{L^\infty}\leq  \frac{2}{\lambda^2}\norm{X_{uu}(\cdot, t)}_{L^\infty}+ b_\lambda e^{6t/\lambda^2}, \quad \text{where } b_\lambda:=\frac{\tanh(\tfrac{1}{4\lambda})\norm{X_u(\cdot,0)}_{L^\infty}^2}{\lambda^3 \ell(0)}
\end{align*}
Then with $V(t)=\norm{X_{uu}(\cdot, t)}_{L^\infty}$,for any $t,t'\in [0,T]$
\begin{align*}
\abs{V(t')-V(t)}  \leq \sup_{u\in \S}\abs{X_{uu}(u,t')-X_{uu}(u,t)}
&\leq \sup_{u\in \S}\int_t^{t'}\abs{\partial_\tau X_{uu}(u, \tau)}d\tau\\
& \leq \abs{t'-t}\max_{t\in [0,T]}\left (\frac{2}{\lambda^2}V(t)+b_\lambda e^{6t/\lambda^2}\right )
\end{align*}
 and therefore $V$ is absolutely continuous. Using Danskin's theorem again, with $U^*:=\{ u\in \S: \abs{X_{uu}(u,t)}=V(t)\}$
\begin{align*}
V'(t)=\max_{u^*\in U^*}\partial_t\abs{X_{uu}(u^*,t)} \leq \max_{u^*\in U^*}\abs{\partial_t X_{uu}(u^*,t)} \leq 
\max_{u\in \S }\abs{\partial_t X_{uu}(u,t)}\leq \frac{2}{\lambda^2}V(t)+b_\lambda e^{6t/\lambda^2} \text{ a.e. }
\end{align*}
Now a Gr\"onwall-type inequality yields a bound on the growth of $V$.
\end{proof}

\begin{proposition}
Assume that $X_0\in C^2(\S,\R^2)$ is an immersion and that $X_0(\S)$ is the boundary of a strictly convex set in $\R^2$.
Then $X(\S,t)$, where $X\in C^1([0,\infty),C^2(\S,\R^2))$ is the solution to $\eqref{GF2}$ with initial data $X_0$, is the boundary of a strictly convex set for all $t$. 
\end{proposition}
\begin{proof} 
Since $X_0(\S)$ is the boundary of a strictly convex set and $X_0\in C^2$, we have $k(u,0)>0$ for all $u\in \S$. We will estimate the time evolution of curvature in order to prove that it remains positive. Using the notation $\partial_s=\frac{1}{\abs{X_u}}\partial_u$, $T=\partial_s X$ and $kN=\partial_s T$, and the commutation relation:
\[
\partial_t\partial_s = \partial_s\partial_t -\langle \partial_s \partial_tX,T \rangle \partial_s
\]
we compute 
\begin{align*}
\partial_t (kN) &=\partial_s \partial_t T-\ip{\partial_s\partial_t X,T} kN\\
&= \partial_s \big ( \partial_s \partial_t X-\ip{\partial_s \partial_t X,T }T\big )-\ip{\partial_s\partial_t X,T} kN\\
&=\partial_s^2 \partial_t X
-\ip {\partial_s^2\partial_t X,T} T
-\langle \partial_s \partial_tX,kN\rangle T
-2\langle \partial_s \partial_t X,T \rangle kN,
\end{align*}
and then 
\[
\partial_t (k^2)=2\langle \partial_t(kN),kN\rangle=2\langle \partial_s^2\partial_t X,kN\rangle - 4k^2\langle \partial_s \partial_t X,T\rangle .
\]
From \eqref{fullFdash} we calculate 
\begin{align*}
    \partial_s \partial_t X & =-\frac{1}{\lambda^2}\left (T+T\ast_X\mathcal G^\lambda\right)\\
    \partial_s^2 \partial_tX&=-\frac{1}{\lambda^2}\left (kN+\frac{1}{\length(X)^2}X\ast_X(\phat^2 \mathcal G^\lambda) \right)=-\frac{1}{\lambda^2}\left (kN+\frac{1}{\lambda^2\length(X)^2}\left (X+X\ast_X\mathcal G^\lambda \right)\right),
\end{align*}
and it follows that 
\[
\partial_t(k^2)=-\frac{2}{\lambda^2}k^2-\frac{2}{\lambda^4\length(X)^2}\langle X+X\ast_X\mathcal G^\lambda,kN\rangle +\frac{4k^2}{\lambda^2}\left (1+\langle T\ast_X\mathcal G^\lambda ,T\rangle \right).
\]
By continuity there exists $t_*>0$ such that $X(\cdot,t)$ remains the boundary of a strictly convex set while $t\in [0,t_*)$. Hence for $t\in [0,t_*)$
we have $k>0$ and
\begin{align}\label{curvatureevo}
\partial_t k&=\frac{k}{\lambda^2}+\frac{2k}{\lambda^2}\langle T\ast_X\mathcal G^\lambda,T\rangle-\frac{1}{\lambda^4\length(X)^2}\langle X+X\ast_X\mathcal G^\lambda,N\rangle
\end{align}
Using \eqref{curlygintegral},
\[
\langle X+X\ast_X\mathcal G^\lambda,N\rangle=\int\langle X(v)-X(u),N(u)\rangle \mathcal{G}^\lambda(\rho_X(u)-\rho_X(v))\hat dv
\]
and then because $\mathcal G<0 $, and the convexity of $X(\S,t)$ gives $\langle X(v)-X(u),N(u)\rangle>0$, we have $\langle X+X\ast_X\mathcal G^\lambda,N\rangle<0 $. Now, with $|\langle T\ast_X\mathcal G^\lambda,T\rangle|\leq 1$, \eqref{curvatureevo} yields
\[
\partial_t k \geq -\frac{k}{\lambda^2},
\]
and therefore $k(u,t)\geq k(u,0)e^{-t/\lambda^2}$ for all $t\in[0,t_*)$ and all $u\in \S$. 
By the continuity of $k$ with respect to $t$, this inequality 
implies $k(\cdot,t)>0$ for all $t\ge0$.
\end{proof}

Theorem \ref{thm:main2} now follows by combining the preceding proposition with the correspondence between solutions of \eqref{GF2} and \eqref{GF} of Proposition \ref{timescaling}.

\appendix

\section{}\label{appendixA}
We prove the claim from Section~\ref{sec:existence} that the constant speed reparametrisation $\hat \gamma(x):=\gamma\circ \varphi_\gamma(x)$ is absolutely continuous with $|\hat \gamma'(x)|=\length(\gamma)$ a.e. The proof is adapted from Lemma 1.1.4 in \cite{AmbrosioGigliSavare2008}, which treats the case of arc-length reparametrisation. 

Note that for any $0\leq x_1\leq x_2\leq 1$, $u_i=\varphi_\gamma(x_i)$, we have
\begin{equation*}
\abs{\hat\gamma(x_1)-\hat\gamma(x_2)}=\abs{\gamma(u_1)-\gamma(u_2)}\leq \int_{u_1}^{u_2}\abs{\gamma'(u)}du=\length(\gamma)(x_2-x_1),
\end{equation*}
showing that $\hat\gamma$ is Lipschitz, hence differentiable a.e., and it follows that  $\abs{\hat\gamma'(x)}\leq \length(\gamma)$. To obtain a lower bound for $\abs{\hat\gamma'(x)}$, consider $0<u_1<u_2<1$ with $x_i=\rho_\gamma(u_i)$, then since $\gamma=\hat\gamma\circ \rho_\gamma$
\[
\abs{\gamma(u_1)-\gamma(u_2)}=\abs{\hat\gamma(x_1)-\hat\gamma(x_2)}\leq \int_{x_1}^{x_2}\abs{\hat\gamma'(x)}dx=\frac{1}{\length(\gamma)}\int_{u_1}^{u_2}\abs{\hat\gamma'(\rho_\gamma(u))}\abs{\gamma'(u)}du.
\]
Dividing each side by $|u_1-u_2|$, taking a limit and applying the Lebesgue differentiation theorem, it follows that 
\[
\length(\gamma) \abs{\gamma'(u)}\leq \abs{\hat \gamma'(\rho_\gamma(u))}\abs{\gamma'(u)} \quad \text{for a.e. } u\in[0,1]
\]
and then $\length(\gamma) \leq \abs{\hat\gamma'(\rho_\gamma(u))}$ a.e. in $[0,1]\setminus P$ where $P:=\{ u\in [0,1]: |\gamma'(u)|=0\}$. Since $\rho_\gamma$ is absolutely continuous it maps null sets to null sets. Moreover, since $\rho_\gamma$ is also monotone, with  $\rho_\gamma'(u)=\abs{\gamma'(u)}/\length(\gamma)$, we have $|\rho_\gamma(P)|\leq \int_P\rho_\gamma'(u)du=0 $. Hence $\length(\gamma)\leq \abs{\hat\gamma'(x)}$ a.e. $x\in [0,1]$, which combined with the upper bound obtained above yields $\abs{\hat\gamma'(x)}=\length(\gamma)$ almost everywhere.

\bibliographystyle{plain}
\bibliography{gradient_flows}

@book{hytonen2016analysis,
  title={Analysis in Banach spaces},
  author={Hyt{\"o}nen, Tuomas and Van Neerven, Jan and Veraar, Mark and Weis, Lutz},
  volume={1},
  year={2016},
  publisher={Springer}
}

@article{danskin1966theory,
  title={The theory of max-min, with applications},
  author={Danskin, John M},
  journal={SIAM Journal on Applied Mathematics},
  volume={14},
  number={4},
  pages={641--664},
  year={1966},
  publisher={SIAM}
}

@book{AmbrosioGigliSavare2008,
  title     = {Gradient Flows: In Metric Spaces and in the Space of Probability Measures},
  author    = {Ambrosio, Luigi and Gigli, Nicola and Savar{\'e}, Giuseppe},
  edition   = {2nd},
  series    = {Lectures in Mathematics ETH Z{\"u}rich},
  publisher = {Birkh{\"a}user},
  address   = {Basel},
  year      = {2008},
  isbn      = {978-3-7643-8721-1}
}

@article{schrader2023h,
	author = {Schrader, Philip and Wheeler, Glen and Wheeler, Valentina-Mira},
	journal = {The Journal of Geometric Analysis},
	number = {9},
	pages = {297},
	publisher = {Springer},
	title = {On the ${H}^1 (ds^\gamma)$-Gradient Flow for the Length Functional},
	volume = {33},
	year = {2023}}

@book{Zeidler:1986aa,
	author = {Zeidler, Eberhard},
	doi = {10.1007/978-1-4612-4838-5},
	isbn = {0-387-90914-1},
	mrclass = {47Hxx (46-02 46N05 54H25 55M20 58-02)},
	mrnumber = {816732},
	mrreviewer = {Jean Mawhin},
	note = {Fixed-point theorems, Translated from the German by Peter R. Wadsack},
	pages = {xxi+897},
	publisher = {Springer-Verlag, New York},
	title = {Nonlinear functional analysis and its applications. {I}},
	url = {https://mathscinet.ams.org/mathscinet-getitem?mr=816732},
	year = {1986}}

@article{A,
	author = {Angenent, Sigurd},
	doi = {10.2307/1971426},
	fjournal = {Annals of Mathematics. Second Series},
	issn = {0003-486X},
	journal = {Ann. of Math. (2)},
	mrclass = {35K15 (35B45 53C42 58E99)},
	mrnumber = {1078266},
	mrreviewer = {Yong-Geun Oh},
	number = {3},
	pages = {451--483},
	title = {Parabolic equations for curves on surfaces. {I}. {C}urves with {$p$}-integrable curvature},
	url = {https://mathscinet.ams.org/mathscinet-getitem?mr=1078266},
	volume = {132},
	year = {1990}}

@article{GH,
	author = {Gage, Michael and Hamilton, Richard S and others},
	journal = {Journal of Differential Geometry},
	number = {1},
	pages = {69--96},
	publisher = {Lehigh University},
	title = {The heat equation shrinking convex plane curves},
	volume = {23},
	year = {1986}}

@article{Gray,
	author = {Grayson, Matthew A},
	journal = {Annals of Mathematics},
	number = {1},
	pages = {71--111},
	publisher = {JSTOR},
	title = {Shortening embedded curves},
	volume = {129},
	year = {1989}}

@article{Gage:1983aa,
	author = {Gage, Michael E.},
	doi = {10.1215/S0012-7094-83-05052-4},
	fjournal = {Duke Mathematical Journal},
	issn = {0012-7094},
	journal = {Duke Math. J.},
	mrclass = {52A40 (53A04)},
	mrnumber = {726325},
	mrreviewer = {R. Osserman},
	number = {4},
	pages = {1225--1229},
	title = {An isoperimetric inequality with applications to curve shortening},
	url = {https://mathscinet.ams.org/mathscinet-getitem?mr=726325},
	volume = {50},
	year = {1983}}

@book{Andrews:2020aa,
	author = {Andrews, Ben and Chow, Bennett and Guenther, Christine Marie and Langford, Mathew},
	publisher = {American Mathematical Society},
	title = {Extrinsic Geometric Flows},
	year = {2020}}

@article{bauer2024elastic,
  title={Elastic metrics on spaces of euclidean curves: Theory and algorithms},
  author={Bauer, Martin and Charon, Nicolas and Klassen, Eric and Kurtek, Sebastian and Needham, Tom and Pierron, Thomas},
  journal={Journal of Nonlinear Science},
  volume={34},
  number={3},
  pages={56},
  year={2024},
  publisher={Springer}
}

@article{MR2201275,
	author = {Michor, Peter W. and Mumford, David},
	doi = {10.4171/JEMS/37},
	fjournal = {Journal of the European Mathematical Society (JEMS)},
	issn = {1435-9855},
	journal = {J. Eur. Math. Soc. (JEMS)},
	mrclass = {58B20 (58D15)},
	mrnumber = {2201275},
	mrreviewer = {Nikolai K. Smolentsev},
	number = {1},
	pages = {1--48},
	title = {Riemannian geometries on spaces of plane curves},
	url = {https://doi.org/10.4171/JEMS/37},
	volume = {8},
	year = {2006}}

@article{bauer2011sobolev,
  title={SOBOLEV METRICS ON SHAPE SPACE OF SURFACES},
  author={Bauer, Martin and Harms, Philipp and Michor, Peter W},
  journal={Journal of Geometric Mechanics},
  volume={3},
  number={4},
  pages={389--438},
  year={2011}
}

\end{document}